\documentclass[10pt]{amsart}
\usepackage{amsmath}
\usepackage{amssymb}
\usepackage{amsfonts}
\usepackage{amsthm}
\usepackage{enumerate}
\usepackage[all]{xy}
\usepackage[latin1]{inputenc}
\usepackage{graphicx}
\usepackage{latexsym}
\input xy

\newtheorem{Theo}{Theorem}[section]
\newtheorem{Prop}[Theo]{Proposition}
\newtheorem{Cor}[Theo]{Corollary}
\newtheorem{Lemma}[Theo]{Lemma}

\theoremstyle{definition}

\newcommand{\rep}{{\rm rep}}
\newcommand{\Rep}{{\rm Rep}}
\newcommand{\Hom}{{\rm Hom}}

\newcommand{\mmod}{{\rm mod}}

\newcommand{\C}{\mathcal{C}}

\begin{document}

\begin{abstract} Let $k$ be a field, $Q$ a quiver with
countably many vertices and $I$ an ideal of $kQ$ such that $kQ/I$
has finite dimensional Hom-spaces. In this note, we prove that there
is no almost split sequence ending at an indecomposable not finitely
presented representation of the bound quiver $(Q,I)$. We then get
that an indecomposable representation $M$ of $(Q,I)$ is the ending
term of an almost split sequence if and only if it is finitely
presented and not projective. The dual results are also true.
\end{abstract}

\title{A non-existence theorem for almost split sequences}

\bigskip
\bigskip

\author{Charles Paquette}
\address{C. Paquette, Dept. of Math., U. de Sherbrooke, 2500 boul. de
l'Université, Sherbrooke, Qc, Canada, J1K 2R1}
\email{charles.paquette@usherbrooke.ca}

\bigskip
\bigskip
\maketitle

\section*{Introduction}

The theory of almost split sequences has been introduced in the
seventies (see \cite{AR2,AR3}) and is essential in the study of the
representation theory of finite dimensional algebras.  The framework
of this theory is general and the almost split sequences can be
studied in other contexts. Indeed, many results concerning the
existence of almost split sequences exist; see for example
\cite{A1,A2,AR2,HuPena,But,Zimm}. However, few is known about the
non-existence of such sequences, at least in categories which are
not Hom-finite. Some criteria following from the definition of left
or right almost split morphisms are discussed in \cite{Hu}. However,
these conditions do not give the description of the objects $C$ for
which there is a right almost split morphism ending in $C$.  The
motivation of this note comes from a work of the author with R.
Bautista and S. Liu in which the Auslander-Reiten theory of the
finitely presented representations over an infinite quiver is
discussed; see \cite{BLP}.

The main categories considered in this note are the categories of
locally finite dimensional representations of infinite bound
quivers. However, some results are stated in more general
categories, so that it could be useful in an other context.  Section
1 is devoted to the main definitions concerning the representations
of quivers. In Section 2, we develop our main tool for non-existence
of almost split sequences.  Section 3 contains the main theorem,
giving necessary and sufficient conditions on an indecomposable
representation $C$ to be the end-term of an almost split sequence.

\section{Background on representations of quivers}

Let $Q=(Q_0,Q_1)$ be a quiver with countably many vertices and $k$
be any field. A typical example of such a quiver is a locally finite
quiver. Let $I$ be an ideal of the path category $kQ$ such that $I
\subseteq Q^2$, where $Q^2$ denotes the ideal of $kQ$ generated by
all paths of length two.  We set $A_{Q,I}=kQ/I$ for the quotient
category of $kQ$ by the ideal $I$. We call $I$ \emph{admissible} if
for $x,y \in Q_0$, $kQ(x,y)/I(x,y)$ is finite dimensional.  In this
case, the pair $(Q,I)$ is called a \emph{bound quiver}.  Observe
that in such a case, the number of arrows between two given vertices
must be finite. If $Q$ is a locally finite quiver such that for any
pair $x,y \in Q_0$, there is a finite number of paths from $x$ to
$y$, then $(Q,I=0)$ is a bound quiver. Also, if $Q$ is finite and
$I$ is admissible in the sense of \cite{ASS}, then $(Q,I)$ is a
bound quiver.

\medskip

Let ${\rm Rep}(Q,I)$ denote the category of all right-modules over
$A_{Q,I}$. Such a module is a covariant functor from $A_{Q,I}$ to
the category ${\rm Mod}(k)$ of all $k$-vector spaces. Hence, $M \in
\Rep(Q,I)$ is given by two families $(M(x))_{x \in Q_0}$ and
$(M(\alpha))_{\alpha \in Q_1}$ where, for $x \in Q_0$, $M(x)$ is a
$k$-vector space and for an arrow $\alpha : x \to y$, $M(\alpha) :
M(x) \to M(y)$ is a $k$-linear map. Moreover, the maps
$(M(\alpha))_{\alpha \in Q_1}$ must satisfy the relations of the
ideal $I$.  If $M(x)$ is finite dimensional for every $x \in Q_0$,
then $M$ is said to be \emph{locally finite dimensional}.  The full
subcategory of $\Rep(Q,I)$ of all such representations is denoted by
$\rep(Q,I)$. Observe that $\rep(Q,I)$ and $\Rep(Q,I)$ are not
Hom-finite in general.  However, from \cite[Section 3.6]{GR}, every
indecomposable object in $\rep(Q,I)$ has a local endomorphism
algebra.

\medskip

Let $\overline{A}$ denote the  algebra associated to $A_{Q,I}$ (with
no identity if $Q$ is infinite), that is,
$$\displaystyle \overline{A}=\oplus_{x,y \in A_{Q,I}}A_{Q,I}(x,y)$$ as $k$-vector spaces and the multiplication is induced by the
composition of morphisms in $A_{Q,I}$. For $x \in A_0$, let $e_x: x
\to x$ denote the identity morphism. Then $e_x$ is a primitive
idempotent in $\overline{A}$. It is easy to see that there exists an
equivalence between $\rep(Q,I)$ and the category $\mmod
(\overline{A})$ of all right $\overline{A}$-modules $M$ such that
$$M = \oplus_{x \in Q_0}Me_x$$ and $Me_x$ is finite dimensional for
any vertex $x$ in  $A_{Q,I}$.  We will make these identifications in
the sequel. Observe that for each $x \in Q_0$, one has an
$\overline{A}$-module $e_x\overline{A}$ which corresponds to a
locally finite dimensional representation $P_x$ of $(Q,I)$.  Note
that $P_x$ is projective indecomposable and has a one dimensional
top. A representation $M$ in $\rep(Q,I)$ is said to be
\emph{finitely generated} if one has an epimorphism
$$P \stackrel{f}{\rightarrow} M \to 0$$
with $P$ isomorphic to a finite direct sum of representations of the
form $P_x$, $x \in Q_0$.  Observe that if $f' : P' \to M$ is any
other such morphism, then ${\rm Ker}\,f'$ is finitely generated if
and only if ${\rm Ker}\,f$ is.  In this case, $M$ is said to be
\emph{finitely presented}.

\medskip

Now, let $(Q^{\rm \, op}, I^{\rm \, op})$ be the opposite bound
quiver of $(Q,I)$, that is, $Q^{\rm op}$ is the opposite quiver of
$Q$ and $I^{\rm \, op}$ is the ideal of $kQ^{\rm\, op}$ such that
$kQ^{\rm\,op}/I^{\rm \, op}$ is the opposite category of $A_{Q,I}$.
Let $D = \Hom(-,k)$ be the duality between finite dimensional
$k$-vector spaces and let $D_Q : \rep(Q,I) \to \rep(Q^{\rm \, op},
I^{\rm \, op})$ denote the pointwise duality defined as follows. If
$M \in \rep(Q,I)$, then one sets $D_QM(x) = DM(x)$ and for $\alpha:
x \to y$, $D_QM(\alpha) : DM(y) \to DM(x)$ is the transpose of the
map $M(\alpha)$.  Now let $f : M \to N$ be a morphism in
$\rep(Q,I)$, that is, a family $\{f_x:M(x) \to N(x)\}_{x \in Q_0}$
of $k$-linear maps such that for each arrow $\alpha : x \to y$,
$N(\alpha)f_x = f_yM(\alpha)$.  We set $D_Q(f)$ to be the morphism
$D_QN \to D_QM$ such that $D_Q(f)_x : DN(x) \to DM(x)$ is the
transpose of the map $f_x$. It is easily verified that $D_Q$ defines
a functor and is a duality. If $P'_x$ is the projective
indecomposable representation in $\rep(Q^{\rm \, op}, I^{\rm \,
op})$ associated to the vertex $x \in Q^{\rm \, op}_0$, then $I_x :=
D_{Q^{\rm \, op}}(P'_x)$ is an indecomposable injective
representation in $\rep(Q,I)$ with a one dimensional socle.  A
representation $M$ in $\rep(Q,I)$ is said to be \emph{finitely
co-generated} if one has a monomorphism
$$0\to M \stackrel{f}{\rightarrow} I$$
with $I$ isomorphic to a finite direct sum of representations of the
form $I_x$, $x \in Q_0$. If, moreover, the cokernel of $f$ if
finitely co-generated, then $M$ is said to be \emph{finitely
co-presented}.  As for finitely presented representations, the
notion of finitely co-presented representation does not depend on
the chosen morphism $f$.

\medskip

Let $\rep^+(Q,I)$ be the full subcategory of $\rep(Q,I)$ of the
finitely presented representations.  If $I=0$, then $\rep^+(Q,I)$ is
abelian by \cite{BLP}.  However, when $I$ is non-zero, $\rep^+(Q,I)$
need not be abelian.  If $M \in \rep^+(Q,I)$ is indecomposable and
not projective, then one has an almost split sequence
$$0 \to \tau M \to E \to M \to 0$$
in $\rep(Q,I)$ with $\tau M$ finitely co-presented; see \cite{A2}.
The dual results also hold.  When $I=0$ and $Q$ is locally finite,
the precise description of the quivers $Q$ for which all such
sequences lie in $\rep^+(Q,I)$ is given in \cite{BLP}. There is a
similar characterization given in \cite{CS} when $(Q,I)$ is such
that $Q$ is locally finite and $I$ is locally finitely generated.
They find the bound quivers $(Q,I)$ such that the category of finite
dimensional representations has almost split sequences.

\medskip

In this paper, however, we will not restrict to the category
$\rep^+(Q,I)$. We shall work in the whole category $\rep(Q,I)$.  We
know that every indecomposable non-projective representation in
$\rep^+(Q,I)$ is the end-term of an almost split sequence in
$\rep(Q,I)$.  We shall show that all other indecomposable
representations in $\rep(Q,I)$ are not end-terms of almost split
sequences.

\section{Left and right almost split morphisms}

Let $\C$ be an abelian $k$-category and let $M\in \C$ be
indecomposable. A morphism $f:E \to M$ in $\C$ is said to be a
\emph{right almost split} if it is not a retraction and any morphism
$L \to M$ in $\C$ which is not a retraction factors through $f$.
Dually, $f:M \to E'$ in $\C$ is said to be a \emph{left almost
split} morphism if it is not a section and any morphism $M \to L$ in
$\C$ which is not a section factors through $f$.  A non-split short
exact sequence
$$0\to L \stackrel{f}{\rightarrow} M \stackrel{g}{\rightarrow} N \to 0 $$
in $\C$ with $L,N$ indecomposable, $f$ left almost split and $g$
right almost split is called an \emph{almost split sequence}.  Such
sequences play a crucial role, for instance, in the study of the
module category of an Artin algebra, see \cite{AR2}, or in a
Hom-finite Krull-Schmidt category (the definition of almost split
sequences there is slightly different since a Krull-Schmidt category
is not necessarily abelian), see \cite{L3}.  For more general facts
concerning almost split sequences, the reader is referred to
\cite{AR3}.

\begin{Prop} \label{mainprop}
Let $N \in \C$ be indecomposable with a chain
$$N_0 \stackrel{f_0}{\longrightarrow} N_1 \stackrel{f_1}{\longrightarrow} N_2 \stackrel{f_2}{\longrightarrow} \cdots$$
of monomorphisms in $\C$ such that the direct limit of the directed
family $\{N_i,f_i\}_{i\ge 0}$ exists and is $\{N,\varphi_i\}_{i\ge
0}$ with the $\varphi_i$ being proper monomorphisms. Suppose
moreover that $\Hom_{\C}(N_i,E)$ is finite dimensional for all $i
\ge 0$ and all objects $E$ in $\C$. Then there is no right almost
split morphism ending in $N$ in $\C$.
\end{Prop}

\begin{proof}
Suppose the contrary.  Let $h: E \to N$ be a right almost split
morphism in $\C$.  Since $\{N,\varphi_i\}_{i\ge0}$ is the direct
limit of $\{N_i,f_i\}_{i\ge 0}$, we have
$\varphi_{i+1}f_{i}=\varphi_i$ for $i\ge 0$. Let $L_i$ denote the
subspace of $\Hom(N_i,E)$ of the morphisms $g$ for which $hg$ is a
multiple of $\varphi_i$. Observe that $L_i$ is finite dimensional. A
morphism $g \in L_i$ for which $hg=\varphi_i$ is called
\emph{normalized}. Since $h$ is right almost split and the
$\varphi_i$ are proper monomorphisms, each $L_i$ contains a
normalized morphism and hence is non-zero. Then one has a non-zero
map
$$g_i = L_{i+1} \to L_i$$
which is induced by $f_i$ and sends a normalized map to a normalized
one.

Since $(*)$ is almost split, $\varphi_i$ yields a normalized map
$v_i : N_i \to E \in L_i$ such that
$$v_if_{i-1}f_{i-2}\cdots f_{j} = g_{j}\cdots g_{i-1}(v_i)$$ is
normalized in $L_{j}$ for $0 \le j < i$.  Let $0\ne M_{ij} = {\rm
Im}(g_jg_{j+1}\cdots g_{i-1})$ for $0 \le j < i$ with $M_{ii} =
L_i$. The chain
$$M_{jj} \supseteq M_{j+1,j} \supseteq M_{j+2,j} \cdots$$
of finite dimensional $k$-vector spaces yields an integer $r_j \ge
j$ for which $0 \ne M_{r_j,j} = M_{k,j}$ whenever $k \ge r_j$.
Moreover, each such $M_{r_j,j}$ contains a normalized map.  Then the
maps $g_i$ clearly induce non-zero maps $$\overline{g}_i :
M_{r_{i+1},i+1} \to M_{r_i,i}.$$  We claim that these maps are
surjective.  Let $u \in M_{r_i,i}$.  For every positive integer
$r>i+1$, $u \in {\rm Im}(g_ig_{i+1} \cdots g_{r-1})$ and hence,
there exists an element $u_r \in {\rm Im}(g_{i+1} \cdots g_{r-1})$
such that ${g_i}(u_r)=u$.  But then $u_{r_{i+1}} \in
M_{r_{i+1},i+1}$ is such that $\overline{g_i}(u_{r_{i+1}})=u$,
showing the claim.  Now, set $u_0 \in M_{r_0,0}$ be a normalized
map.  Then there exists $u_1 \in M_{r_1,1}$ such that
$\overline{g_0}(u_1) = u_0$.  Observe that if $u_1$ is not
normalized, then there exists $\alpha \in k\backslash\{0\}$ such
that $\alpha u_1$ is normalized and hence that
$\overline{g_0}(\alpha u_1)=\alpha u_0$ is normalized, showing that
$\alpha = 1$.  Hence, $u_1$ is normalized.  Choose such $u_i \in
M_{r_i,i}$ for all positive integers $i$.  Hence, for $i \ge 0$, we
have that $hu_i=\varphi_i$ and $u_{i+1}f_i=u_i$. Since $N$ is the
direct limit of the $N_i$, the family of morphisms $u_i : N_i \to E$
yields a unique morphism $h' : N \to E$ such that $h'\varphi_i =
u_i$ for $i\ge 0$.  Now, $hh'\varphi_i = hu_i = \varphi_i$ for all
$i$, showing, by uniqueness, that $hh'=1_N$, contradicting the fact
that $h$ is right almost split.
\end{proof}

We state the dual result.  The proof is dual.

\begin{Prop}
Let $N \in \C$ be indecomposable with a chain
$$N_0 \stackrel{f_0}{\longleftarrow} N_1 \stackrel{f_1}{\longleftarrow} N_2 \stackrel{f_2}{\longleftarrow} \cdots$$
of epimorphisms in $\C$ such that the inverse limit of the directed
family $\{N_i,f_i\}_{i\ge 0}$ exists and is $\{N,\varphi_i\}_{i\ge
0}$ with the $\varphi_i$ being proper epimorphisms. Suppose moreover
that $\Hom_{\C}(E,N_i)$ is finite dimensional for all $i \ge 0$ and
all objects $E$ in $\C$. Then there is no left almost split morphism
starting in $N$ in $\C$.
\end{Prop}

\section{The left and right almost split morphisms in $\rep(Q,I)$}

We start with this lemma.

\begin{Lemma} \label{firstlemma}
Let $M \in {\rm rep}(Q,I)$.  Then
\begin{enumerate}[$(1)$]
    \item The representation $M$ is non-finitely generated if and only if there exists a
chain
$$M_0 \subseteq M_1 \subseteq M_2 \subseteq \cdots$$
of finitely generated proper subrepresentations of $M$ such that
$$\bigcup_{i \ge 0}M_i = M.$$
    \item The representation $M$ is non-finitely co-generated if and only if $D_Q(M)$ is non-finitely generated.
\end{enumerate}

\end{Lemma}

\begin{proof}
Since $Q$ has a countable number of vertices, there exists a chain
$$E_0 \subseteq E_1 \subseteq E_2 \subseteq \cdots$$
of finite subsets of $Q_0$ such that $\cup_{i \ge 0}E_i = Q_0$. For
$i \ge 0$, let $M_i$ be the subrepresentation of $M$ generated by
the elements in $\oplus_{v \in E_i}Me_v$.  Then $M_i$ is finitely
generated such that $\cup_{i\ge 0}M_i =M$.  If $M$ is non-finitely
generated, then it is clear that the $M_i$ are proper
subrepresentations. Now, let $M$ be finitely generated and let
$$M_0 \subseteq M_1 \subseteq M_2 \subseteq \cdots$$
be a chain as in the statement. Since the top of $M$ is finite
dimensional and the union of the $M_i$ is $M$, there exists an
integer $j$ such that $M_j(x)=M(x)$ for any vertex $x \in Q_0$ which
supports the top of $M$.  This yields $M=M_j$, contradicting the
fact that the $M_i$ are proper subrepresentations.  The second
statement trivially follows from the fact that $D_Q$ is a duality.
\end{proof}

\begin{Lemma}
Let $M,N \in {\rm rep}(Q,I)$.  Then
\begin{enumerate}[$(1)$]
    \item If $M$ is finitely generated,
then $\Hom(M,N)$ is finite dimensional.
    \item If $N$ is finitely co-generated,
then $\Hom(M,N)$ is finite dimensional.
\end{enumerate}
\end{Lemma}

\begin{proof}
Suppose that $M$ is finitely generated. Then there exists an
epimorphism
$$P \to M \to 0$$
in $\rep(Q,I)$ with $P$ a projective representation which is a
finite direct sum of representations of the form $P_x$, $x\in Q_0$.
Since we have an inclusion $\Hom(M,N) \to \Hom(P,N)$, we only need
to show that $\Hom(P,N)$ is finite dimensional, which amount to the
same as showing that $\Hom(P_x,N)$ is finite dimensional for $x \in
Q_0$.  However, this is easy to show that $\Hom(P_x,N) \cong N(x)$
as $k$-vector spaces, which shows the claim.  Now, if $N$ is
finitely co-generated, then $D_Q(N)$ is finitely generated in
$\rep(Q^{\rm \, op},I^{\rm \, op})$ and $\Hom(M,N) \cong
\Hom(D_Q(N),D_Q(M))$ as $k$-vector spaces.  But by (1), the last
Hom-space is finite dimensional.
\end{proof}

Observe that the category $\Rep(Q,I)$ of all representations of the
bound quiver $(Q,I)$ is abelian, has co-products and satisfies
$$(\cup_{i\in I}A_i) \cap B = \cup_{i \in I} (A_i \cap B)$$
for $A,B \in \Rep(Q,I)$ and a directed family $\{A_i\}_{i \in I}$ of
subrepresentations of $A$.  This is called a \emph{$C_3$-category},
see \cite{Mi}. Using Lemma \ref{firstlemma}, we get the following.

\begin{Cor} \label{cor1}
Let $M \in \rep(Q,I)$ be indecomposable.
\begin{enumerate}[$(1)$]
    \item If $M$ is not finitely generated,
then there is no right almost split morphism ending in $M$ in
$\rep(Q,I)$. In particular, there is no almost split sequence ending
in $M$ in $\rep(Q,I)$.
    \item If $M$ is not finitely co-generated,
then there is no left almost split morphism starting in $M$ in
$\rep(Q,I)$. In particular, there is no almost split sequence
starting in $M$ in $\rep(Q,I)$.
\end{enumerate}
\end{Cor}

\begin{proof}
We need only to prove the first statement since the second is dual.
Let $M \in \rep(Q,I)$ be non-finitely generated. From Lemma
\ref{firstlemma}, there exists a chain
$$M_0 \subseteq M_1 \subseteq M_2 \subseteq \cdots$$
of finitely generated proper subrepresentations of $M$ such that
$$\bigcup_{i \ge 0}M_i = M.$$
From \cite[Page 82, Prop. 1.2]{Mi}, $\{M,\varphi_i\}_{i\ge 0}$ is
the direct limit of the $M_i$ with $\varphi_i : M_i \to M$ being the
inclusion morphism.  But then $\{M,\varphi_i\}_{i\ge 0}$ is also the
direct limit of the $M_i$ in $\rep(Q,I)$ since $M \in \rep(Q,I)$ and
$\rep(Q,I)$ is a full subcategory of $\Rep(Q,I)$.  Then we can apply
Proposition \ref{mainprop}.
\end{proof}

Now, we need to look at representations which are finitely generated
but not finitely presented.  Before stating the next proposition, we
need to introduce some definitions.  The following could be found in
\cite{War} in the context of a $C_3$-category. However, we apply
them for the category $\Rep(Q,I)$.  A representation $A \in
\Rep(Q,I)$ is said to be \emph{small} if whenever $f : A \to
\oplus_{i \in I}A_i$ is a morphism with $p_i : \oplus_{i \in I}A_i
\to A_i$ the canonical projections, then $p_if=0$ for all but a
finite number of $i \in I$. It is easy to see that for $x \in Q_0$,
$P_x$ is small in $\Rep(Q,I)$ and consequently, every finitely
generated representation of $\rep(Q,I)$ is small in $\Rep(Q,I)$.  A
representation $A \in \Rep(Q,I)$ is said to be \emph{$\sigma$-small}
is it is the union of a chain
$$A_0 \subseteq A_1 \subseteq A_2 \subseteq \cdots$$
of small objects of $\Rep(Q,I)$.  By Lemma \ref{firstlemma}, every
object in $\rep(Q,I)$ is $\sigma$-small. Suppose that $E \in
\rep(Q,I)$ decomposes as a finite direct sum of indecomposable
representations.  In particular, by \cite[Section 3.6]{GR}, it
decomposes as a finite direct sum of $\sigma$-small representations
with local endomorphism algebras.  By \cite[Theorem 7]{War}, every
other decomposition of $E$ refines to this given decomposition. This
fact will be useful in the proof of the following.

\begin{Prop}
Let $M \in \rep(Q,I)$ be indecomposable.
\begin{enumerate}[$(1)$]
    \item If $M$ is finitely generated but not finitely
presented, then there is no almost split sequence ending in $M$ in
$\rep(Q,I)$.  In particular, there is no right almost split morphism
$E \to M$ in $\rep(Q,I)$ with $E$ a finite direct sum of
indecomposable representations.
    \item If $M$ is finitely co-generated but not finitely
co-presented, then there is no almost split sequence starting in $M$
in $\rep(Q,I)$.  In particular, there is no left almost split
morphism $M \to E$ in $\rep(Q,I)$ with $E$ a finite direct sum of
indecomposable representations.
\end{enumerate}
\end{Prop}

\begin{proof}
Again, we only prove the first assertion. Suppose the contrary.  Let
$$(*): 0 \to L \to E \stackrel{h}{\rightarrow} M\to 0$$
be an almost split sequence in $\rep(Q,I)$ with $M$ finitely
generated but not finitely presented. Let
$$0\to \Omega \to P \to M \to 0$$
be a short exact sequence with $P$ a finite direct sum of
representations of the form $P_x$, $x \in Q_0$.  Then, by the
hypothesis, $\Omega$ is not finitely generated. In particular, being
a subrepresentation of a finitely generated representation, it has
an infinite dimensional top. Let $\{u_i\}_{i \ge 1}$ be an infinite
family in $Q_0$ such that the top of $\Omega$ has support
$\{u_i\}_{i\ge 1}$. For each $i$, let $S_i$ be a simple quotient of
$\Omega$ with the support of $S_i$ being $\{u_i\}$. Consider the
pushout diagram
$$\xymatrix@1{0 \ar[r] & \Omega \ar[r] \ar[d] & \,P \ar[d] \ar[r] & \,M \ar@{=}[d] \ar[r] & \,0\\
0 \ar[r] & S_i \ar[r]^g & \,E_i \ar[r]^f & \,M\ar[r] & \,0}$$ The
last row is not split since $\Omega \to P$ is a radical morphism.
Hence, one has the following pushout diagram by using the fact that
$(*)$ is almost split.
$$\xymatrix@1{0 \ar[r] & S_i \ar[r]^g \ar[d]^v & \,E_i \ar[d]^u \ar[r]^f & \,M \ar@{=}[d] \ar[r] & \,0\\
0 \ar[r] & L \ar[r] & \,E \ar[r]^h & \,M \ar[r] & \,0}$$ Observe
that $ug \ne 0$ since otherwise, $u$ induces a map $u': M \to E$
such that $u=u'f$.  But then, $hu'f = f$ yielding $hu' = 1_M$ since
$f$ is an epimorphism.  This is a contradiction. Therefore, $v$ is
non-zero and hence is a monomorphism. This shows that $L$ is not
finitely co-generated, since it has an infinite dimensional socle.
Hence, there is no left almost split morphism starting in $L$ using
part (2) of Corollary \ref{cor1}. But this contradicts the fact that
$(*)$ is almost split.

Now, suppose that $f : E \to M$ is right almost split with $E$ a
finite direct sum of indecomposable representations. Since, $M$ is
not finitely presented, $M$ is not projective and hence, $f$ is an
epimorphism. Let $h$ be an endomorphism of $E$ such that $fh=f$. By
\cite[Section 3.8]{GR}, $h$ induces a decomposition $E=E_1 \oplus
E_2$ of $E$ such that $f$ is stable on $E_1$ and $E_2$ and $h$ is
pointwise nilpotent on $E_1$, meaning that for each $x \in Q_0$,
there exists a positive integer $n$ such that $(h_x)^n(E_1)=0$.
Moreover, $h$ is an isomorphism on $E_2$.  Let $h_1 : E_1 \to E_1$
and $h_2 : E_2 \to E_2$ be the restriction of $h$ to $E_1$ and
$E_2$, respectively. We have $h=h_1 \oplus h_2$. Let $f=(f_1 \;
f_2)$ be the corresponding decomposition of $f$. Since $h_1$ is
locally nilpotent, we have that, for $x \in Q_0$, there exists a
positive integer $r$ such that $(f_1)_x = (f_1)_x(h_1)_x^r=0$.
Hence, $f_1=0$, meaning that $f_2$ is a right almost split
epimorphism. Now, if every endomorphism $h : E \to E$ with $fh=f$ is
an automorphism, then $f$ is a right minimal almost split
epimorphism (see \cite{AR3}) and
$$0 \to {\rm Ker}f \to E \to M \to 0$$
is then an almost split sequence by using classical arguments on
almost split sequences. This is a contradiction. Hence, there exists
a non-automorphism $h: E \to E$ providing a non-trivial
decomposition $E= E_1 \oplus E_2$ as above with $f_2 : E_2 \to M$
being a right almost split epimorphism. Now, $f_2$ is not minimal
since this would yield an almost split sequence ending in $M$.
Hence, as argued above, we can decompose $E_2$ non-trivially as $E_2
= E_3 \oplus E_4$ with the restriction $E_4 \to M$ of $f_2$ being a
right almost split epimorphism. We can continue this process
infinitely many times, decomposing $E$ as an infinite direct sum of
subrepresentations of $E$. However, this is a contradiction by the
remark preceding the lemma.
Hence, there is no right almost split morphism $E \to M$ with $E$
being a finite direct sum of indecomposable representations.
\end{proof}

As mentioned in Section 1, it is proved in \cite{A2} that there
exists an almost split sequence in $\rep(Q,I)$ ending at each
indecomposable finitely presented and non-projective representation
$M$.  One also has the dual result. Therefore, we get the following
main theorem by combining the results obtained so far.

\begin{Theo}
Let $M \in \rep(Q,I)$ be indecomposable.
\begin{enumerate}[$(1)$]
    \item There is an almost split sequence ending in $M$ in
    $\rep(Q,I)$ if and only if $M$ is finitely presented and
    non-projective.
    \item There is an almost split sequence starting in $M$ in
    $\rep(Q,I)$ if and only if $M$ is finitely co-presented and
    non-injective.
\end{enumerate}
\end{Theo}

\bigskip

{\em Acknowledgments} The author is thankful to S. Liu and V.
Shramchenko for financial support while doing a postdoctorate at the
University of Sherbrooke.


\begin{thebibliography}{99}


\bibitem{ASS} I. Assem, D. Simson and A. Skowro\'nski, {\em Elements of representation theory of associative algebras. Vol.
1.}, Techniques of representation theory. London Mathematical
Society Student Texts {\bf 65} Cambridge University Press,
Cambridge, 2006.

\bibitem{A1} M. Auslander {\em Almost split sequences and algebraic
geometry}, Representations of algebras, Durham, 1985, 165--179,
London Math. Soc. Lecture Note Ser. {\bf 116} Cambridge University
Press, Cambridge, 1986.

\bibitem{A2}  M. Auslander, {\em A survey of existence theorems for almost split
sequences}, Representations of algebras, London Math. Soc. Lecture
Note Ser. {\bf 116} Cambridge University Press, Cambridge, 1986,
81--89.

\bibitem{AR2} M. Auslander and I. Reiten, {\em Representation theory of Artin algebras. III. Almost split
sequences}, Comm. Algebra {\bf 3} (1975), 239--294.

\bibitem{AR3} M. Auslander and  I. Reiten, {\em Representation theory of artin
algebras IV}, Comm. Alebra {\bf 5} (1977) 443--518.


\bibitem{Hu} L. Angeleri-H$\ddot{u}$gel,
{\em On the existence of left and right almost split morphisms},
Representations of algebras, São Paulo, 1999, 1--9, Lecture Notes in
Pure and Appl. Math. {\bf 224} Dekker, New York, 2002.

\bibitem{HuPena}L. Angeleri-H$\ddot{u}$gel and J.A. de la Peña, {\em Locally finitely generated modules over rings
with enough idempotents}, J. Algebra Appl. {\bf 8} (2009), no. 6,
885--901.

\bibitem{BLP}  R. Bautista, S. Liu and C. Paquette, {\em Auslander-Reiten theory over an infinite quuiver}, preprint.

\bibitem{But} W. Burt and M.C.R. Butler {\em Almost split sequences for bocses}, Representations of finite-dimensional algebras, Tsukuba, 1990, 89--121,
CMS Conf. Proc. {\bf 11} Amer. Math. Soc., Providence, RI, 1991.

\bibitem{CS} F. Coelho and S. Smalø, {\em Almost split sequences in categories of quivers with
relations}, Representations of algebra. Vol. I, II, 183--191,
Beijing Norm. Univ. Press, Beijing, 2002.

\bibitem{GR} P. Gabriel et A.V. Roiter, {\em Representations of finite
dimensional algebras},  Algebra VIII, Encyclopedia Math. Sci. {\bf
73} Springer, Berlin, 1992.

\bibitem{Mi} B. Mitchell, {\em Theory of categories}, Pure and Applied Mathematics, Vol. XVII, Academic Press, New York-London, 1965, xi+273 pp.

\bibitem{L3} S. Liu, {\em Auslander-Reiten theory in a Krull-Schmidt
category}, Proceedings of ICRA XIII (S$\tilde{a}$o Paulo), to
appear.

\bibitem{War} R.B. Warfield Jr. {\em Decompositions of injective modules}, Pacific J. Math. {\bf 31} (1969), 263--276.

\bibitem{Zimm} W. Zimmermann {\em Existenz von
Auslander-Reiten-Folgen}, Arch. Math. (Basel) {\bf 40} (1983), no.
1, 40--49.
\end{thebibliography}
\end{document}